\documentclass[11pt]{article}
\usepackage{amsfonts}
\usepackage{amssymb,amsmath}
\usepackage{latexsym}
\usepackage{graphics}
\usepackage{epsfig}
\usepackage{psfrag}
\usepackage{bbm}

\numberwithin{equation}{section}
\def\PP{\mathbb{P}}
\def\QQ{\mathbb{Q}}
\def\RR{\mathbb{R}}
\def\R{\mathbb{R}}

\def\EE{\mathbb{E}}
\def\eps{\varepsilon}
\newcommand\ind{\mathbb{I}}

\newtheorem{thm}{Theorem}[section]
\newtheorem{lem}[thm]{Lemma}
\newtheorem{cor}[thm]{Corollary}

\newtheorem{prop}[thm]{Proposition}

\newenvironment{proof}{\paragraph{Proof} \phantom{9}}{\hfill$\Box$\bigskip}

\begin{document}

\title{Strong solutions to stochastic differential equations with rough coefficients}

\author{Nicolas Champagnat$^{1,2,3}$, Pierre-Emmanuel Jabin$^{4}$}

\footnotetext[1]{Universit\'e de Lorraine, Institut Elie Cartan de Lorraine,
    UMR 7502, Vand\oe uvre-l\`es-Nancy, F-54506, France;
    E-mail:~\texttt{Nicolas.Champagnat@inria.fr}}

\footnotetext[2]{CNRS, Institut Elie Cartan de Lorraine, UMR
    7502, Vand\oe uvre-l\`es-Nancy, F-54506, France}

\footnotetext[3]{Inria, TOSCA, Villers-l\`es-Nancy, F-54600, France}

\footnotetext[4]{Cscamm and Dpt. of Mathematics, University of Maryland, College Park, MD
  20742 USA, E-mail:~\texttt{pjabin@umd.edu}}

\maketitle

\begin{abstract}
  We study strong existence and pathwise uniqueness for stochastic differential equations in
  $\RR^d$ with rough coefficients, and without assuming uniform ellipticity for the diffusion
  matrix. Our approach relies on direct quantitative estimates on solutions to the SDE,
  assuming Sobolev bounds on the drift and diffusion coefficients, and $L^p$ bounds for the
  solution of the corresponding Fokker-Planck PDE, which can be proved separately. This allows
  a great flexibility regarding the method employed to obtain these last bounds. Hence we are
  able to obtain general criteria in various cases, including the uniformly elliptic case in
  any dimension, the one-dimensional case and the Langevin (kinetic) case.
\end{abstract}

\bigskip


\section{Introduction}
%
We investigate the well posedness of the Stochastic Differential Equation (SDE) in $\RR^d$, $d\geq 1$,
\begin{equation}
dX_t=F(t,X_t)\,dt+\sigma(t,X_t)\,dW_t,\quad X_0=\xi,\label{sde}
\end{equation}
where $F:\RR_+\times\R^d\rightarrow \R^d$ and
$\sigma:\mathbb{R}_+\times\mathbb{R}^d\rightarrow \mathbb{R}^d\times\mathbb{R}^r$ are Borel
measurable function, $(W_t,t\geq 0)$ is a $r$-dimensional standard Brownian motion on some
given complete filtered probability space $(\Omega,({\cal F}_t)_{t\geq 0},\mathbb{P})$, and
$\xi$ is a ${\cal F}_0$-measurable random variable.

When $\sigma$ and $F$ are bounded, it is standard to deduce from It\^o's formula that the law $u(t,dx)$ of $X_t$ is a (weak, measure)
solution to the Fokker-Planck PDE on $\RR_+\times\RR^d$
\begin{equation}
  \partial_tu+\nabla_x \cdot(Fu)
  =\nabla^2_x:(au)=\sum_{1\leq i,j\leq d}\frac{\partial^2(a_{ij}u)}{\partial x_i\partial
    x_j},\quad u(t=0,dx)=u^0,\label{pde}
\end{equation}
where $a=\frac{1}{2}\sigma\,\sigma^*$ and $u^0$ is the law of the initial r.v. $\xi$. 
 
We first recall some classical terminology: weak existence holds for~\eqref{sde} if one can construct a filtered probability space
$(\Omega,({\cal F}_t)_{t\geq 0},\mathbb{P})$, an adapted Brownian motion $W$ and an adapted process $X$ on this space solution
to~\eqref{sde}. Uniqueness in law holds if every solution $X$ to~\eqref{sde}, possibly on
different probability space, has the same law, in particular if there is uniqueness of measured-valued solutions to \eqref{pde}.
Strong existence means that one can find a solution to~\eqref{sde} on any given filtered probability space equipped with any given
adapted Brownian motion. Finally, pathwise uniqueness means that, on any given filtered probability space equipped with any given
Brownian motion, any two solutions to~\eqref{sde} with the same given $\mathcal{F}_0$-measu\-rable initial condition $\xi$
coincide. Our goal is to study {\em strong existence and pathwise uniqueness} for rough $\sigma$ and $F$, through quantitative estimates on the difference between solutions.

This question has been the object of many works aiming to improve the original result of It\^o~\cite{Ito}. Krylov and
Veretennikov~\cite{VK76,V82} studied the case of uniformly continuous $a$ and bounded $F$, proving that only two cases are possible:
either pathwise uniqueness holds, or strong existence does not hold. The question was studied again recently by Krylov and
R\"ockner~\cite{KR05} and Zhang~\cite{Zh05,Zh11}. All these works assume that the matrix $a$ is uniformly elliptic, i.e.\ that
$a(x)-c\,\text{Id}$ is positive definite for all $x$ for some constant $c>0$. The time-independent one-dimensional case was also deeply
studied by Engelbert and Schmidt~\cite{ES1} (see also~\cite{YW71,LG}).

The main tools used in all the previous works are Krylov's inequality~\cite{K71} and its extensions (see for
example~\cite{EK00,KR05,KL08,Zh11}), Zvonkin's transformation~\cite{Zv74} to remove the drift, and a priori estimates on solutions of
the backward Kolmogorov equation or Fokker-Planck PDE~\eqref{pde}~\cite{VK76,K01,KR05,Zh05}. Of great importance is also the result
of Yamada and Watanabe~\cite{YW71}, which proves that strong existence holds as soon as pathwise uniqueness and weak existence hold
for all initial condition. Since general conditions for weak existence are well-known (see~\cite{K74,SV79,RS91,EK00,KL08,F08}; see
also~\cite{PLL12} for a recent and deep study of the question), one only has to prove pathwise uniqueness to obtain strong existence.
In dimension one, a key tool to prove pathwise uniqueness is the local time. 

Another approach to strong existence and pathwise uniqueness was recently initiated by Le Bris and Lions in~\cite{LBL04,LBL08}, based
on well-posedness results for the backward Kolmogorov equation. The authors define the notion of almost everywhere stochastic flows
for~\eqref{sde}, which combines existence and a flow property for almost all initial conditions, and give precise results in the case
where $a=\text{Id}$. The general case was recently studied deeply by P.-L. Lions in~\cite{PLL12}, who reduces the question to
well-posedness, $L^1$ norms and stability properties for two backward Kolmogorov equations; the first one associated to the
SDE~\eqref{sde} and the other one obtained by a doubling of variable technique. Note that this approach does not require assumptions
of uniform ellipticity for $a$. 

We present here another approach which relies on estimates on path functionals of the difference between solutions to regularizations
of~\eqref{sde}. This is inspired by the method used by Crippa and De Lellis~\cite{CD} to obtain an alternative proof of the
results of Di Perna and Lions~\cite{DL} on well-posedness for ODEs. The functional of~\cite{CD} was used and adapted to
obtain several extensions~\cite{Ja,CJ} for deterministic systems  (see
also~\cite{RZ10} for a study of weak uniqueness using this method). Note that other techniques exist to prove well posedness directly on characteristics of ODEs, see \cite{HLL} for instance.

The quantitative estimates which we develop here let us treat separately the strong existence
and pathwise uniqueness for \eqref{sde} from the question of bounds on solutions to \eqref{pde}. 
The typical result presented here will hence assume that some estimate could be obtained on
solutions to \eqref{pde} (by whichever method) and conclude that strong existence and pathwise
uniqueness hold provided that
some bounds on $\sigma$ and $F$ in Sobolev spaces related to the bounds on $u$ hold. The great advantage is the flexibility that one then enjoys as it is possible to choose the best method to deal with \eqref{pde} according to any additional structure. For instance, ellipticity on $\sigma$ is not required {\em a priori}.  
The second advantage of the method is its simplicity as it relies on some direct quantitative estimates on the solutions.

To give a better idea let us present a typical result that we obtain. For existence we consider sequence of approximations to \eqref{sde}
\begin{equation}
dX_t^n=F_n(t,X_t^n)\,dt+\sigma_n(t,X_t^n)\,dW_t,\quad X_0^n=\xi,\label{sden}
\end{equation}
with the same Brownian motion $W_t$ for any $n$. And we introduce the corresponding approximation for \eqref{pde}
\begin{equation}
  \partial_tu_n+\nabla_x \cdot(F_nu_n)
  =\sum_{1\leq i,j\leq d}\frac{\partial^2}{\partial x_i\partial
    x_j}(a_{ij}^n(t,x)u_n(t,x)),\quad u_n(t=0,dx)=u^0,\label{pden}
\end{equation}
with $a^n=\sigma_n\,\sigma_n^*$.

The next result is not the most general we obtain, but it does not require any additional definition and illustrates the type of assumptions we need.
\begin{thm} Assume $d\geq 2$. One has
  \begin{description}
  \item[\textmd{(i)}] Existence: Assume that there exists a sequence of smooth $F_n,\;\sigma_n\in L^\infty$ such that the solution
    $u_n$ to \eqref{pden} satisfies for $1\leq p,q\leq\infty$, with $1/p'+1/p=1$, $1/q+1/q'=1$ 
    \[\begin{split}
      & \sigma_n\longrightarrow \sigma \ \mbox{in}\ L^{q}_{t,loc}(L^{p}_x)\quad\text{\ and\
      }\quad F_n\rightarrow F\ \mbox{in}\ L^{q}_{t,loc}(L^{p}_x)
      ,\\
      & \sup_n \left(\|\nabla\sigma_n\|_{L^{2q}_{t,loc}(L^{2p}_x)}+\|\nabla F_n\|_{L^{q}_{t,loc}(L^{p}_x)}+\|F_n\|_{L^\infty}+\|\sigma_n\|_{L^\infty}\right)<\infty,\\
      & \sup_n\|u_n\|_{L^{q'}_{t,loc}(L^{p'}_{x})}<\infty,\ u_n\longrightarrow u\ \mbox{in the weak-* topology of measures.}\\
    \end{split}\] 
    Then there exists a strong solution $X_t$ to \eqref{sde} and $(X^n_t-\xi,t\in[0,T])_n$ converges in $L^p(\Omega,L^\infty([0,T]))$
    for all $p>1$ to $(X_t-\xi,t\in[0,T])$, with $X_t^n$ the solutions to \eqref{sden}. In addition, $u(dt,dx)=u(t,dx)dt$, where
    $u(t,\cdot)$ is the law of $X_t$ for all $t\in[0,T]$.
  \item[\textmd{(ii)}] Uniqueness: Let $X$ and $Y$ be two solutions to \eqref{sde} with one-dimensional time marginals $u_X(t,x)dx$
    and $u_Y(t,x)dx$ both in $L^{q'}_{t,loc}(L^{p'}_x)$. Assume that $X_0=Y_0$ a.s.\ and that
    $$
    \|F\|_{L^q_{t,loc}(W^{1,p}_x)}+\|\sigma\|_{L^{2q}_{t,loc}(W^{1,2p}_x)}<\infty
    $$
    with $1/p+1/p'=1$ and $1/q+1/q'=1$. Then one has pathwise uniqueness: $sup_{t\geq 0} |X_t-Y_t|=0$ a.s.
  \end{description}
  \label{maincorollary}
\end{thm}

We obtain better results in the one-dimensional case.
\begin{thm} Assume $d=1$.
  \begin{description}
  \item[\textmd{(i)}] The existence result of Theorem~\ref{maincorollary}~(i) holds under the same assumptions on $F_n,\sigma_n,u_n$,
    except that the assumption $\sup_n\|\nabla\sigma_n\|_{L^{2q}_{t,loc}(L^{2p}_x)}<\infty$ can be replaced by
    $$
    \sup_n\|\sigma_n\|_{L^{2q}_{t,loc}(W^{1/2,2p}_x)}<\infty
    $$
    and in the case $p=1$, the assumption $\sup_n\|\nabla F_n\|_{L^{q}_{t,loc}(L^{p}_x)}<\infty$ must be replaced by
    $$
    \sup_n\|\nabla F_n\|_{L^{q}_{t,loc}(L^{1+\varepsilon}_x)}<\infty
    $$
    for some $\varepsilon>0$.
  \item[\textmd{(ii)}] The uniqueness result of Theorem~\ref{maincorollary}~(ii) holds true under the same assumptions on
    $F,\sigma,u_X,u_Y$, except that  $\|\sigma\|_{L^{2q}_{t,loc}(W^{1,2p}_x)}<\infty$ can be replaced by
    $$
    \|\sigma\|_{L^{2q}_{t,loc}(W^{1/2,2p}_x)}<\infty.
    $$
    and in the case $p=1$, the assumption $\|F\|_{L^{q}_{t,loc}(W^{1,p}_x)}<\infty$ must be replaced by
    $$
    \|F\|_{L^{q}_{t,loc}(W^{1,1+\varepsilon}_x)}<\infty
    $$
    for some $\varepsilon>0$.
  \end{description}
  \label{maincorollary-1d}
\end{thm}

Note that no assumption of uniform ellipticity is needed in Theorems~\ref{maincorollary} and~\ref{maincorollary-1d}, provided one can
prove a priori estimates on the various solutions $u_n,\ u_X,\ u_Y$ to~\eqref{pden} and~\eqref{pde}. Note also that pathwise uniqueness is
proved only for particular solutions to~\eqref{sde}, so we cannot use directly the result of Yamada and Watanabe to deduce strong
existence. Hence our method proves separately strong existence and pathwise uniqueness; however they use very similar techniques.

The goal of Section~\ref{sec:results} is to give the statement of all our results. We start in
Subsection~\ref{sec:norms-banach-spaces} by defining the norms and Banach spaces needed to state our most general results in
Subsection~\ref{sec:gal-res}. Theorems~\ref{maincorollary} and~\ref{maincorollary-1d} will then be obtained as corollaries of these
general results.

Of course, as they are laws, $u_n$, $u_X$ and $u_Y$ all have bounded mass so Theorems~\ref{maincorollary} and~\ref{maincorollary-1d}
really depend on whether it is possible to obtain higher integrability for a solution of~\eqref{pde}. Several situations where this
can be done will be studied in Subsection~\ref{sec:csq}, including the uniformly elliptic,
non-degenerate one-dimensional and kinetic (Langevin)
cases. The conditions for strong existence and pathwise uniqueness can then be compared with the best conditions in the literature.
The rest of the paper is devoted to the proofs of all the results stated in Section~\ref{sec:results}, and the organization of the
rest of the paper is given in the end of Section~\ref{sec:results}.




\section{Statement of the results}
\label{sec:results}

As usual one needs regularity assumptions on $F$ and $\sigma$ to ensure strong existence and pathwise uniqueness for~(\ref{sde}). In
our case, these are Sobolev norms with respect to $u$, defined in Subsection~\ref{sec:norms-banach-spaces}. Our general results are then
stated in Subsection~\ref{sec:gal-res}, and several consequences of these results are discussed in Subsection~\ref{sec:csq}.

\subsection{Norms and Banach spaces}
\label{sec:norms-banach-spaces}

We fix $T>0$ and pose
\begin{align}
  \|\sigma\|_{H^1(u)}^2&=\int_0^T\int_{\R^d} |\sigma(t,x)|^2\,u(t,dx)\,dt+\int_0^T\int_{\R^d}  (M|\nabla \sigma|(t,x))^2\,u(t,dx)\,dt \notag\\
  &=\mathbb{E}\left(\int_0^T |\sigma|^2(t,X_t)dt\right)+\mathbb{E}\left(\int_0^T (M|\nabla \sigma|(t,X_t))^2dt\right), \label{H1}
\end{align}
where $M$ is the usual maximal operator
\[
Mf(x)=\sup_r \frac{1}{|B(0,r)|}\int_{B(0,r)} f(x+z)\,dz. 
\] 
Of course here $|\nabla\sigma|$ and hence $M|\nabla\sigma|$ could have $+\infty$ values on sets of positive Lebesgue measure. 

Note that this indeed defines a norm which enjoys the usual properties, semi-continuity for instance, as proved in Section~\ref{sec:technical}.
\begin{prop} The definition \eqref{H1} is a norm. Moreover if $\sigma_n\longrightarrow \sigma$ in the sense of distribution then
\[
\|\sigma\|_{H^1(u)}\leq \liminf_n \|\sigma_n\|_{H^1(u)},
\]
And if for a given $\sigma$, $u_n\geq 0$ converges to $u$ for the weak-* topology of measures then
\[
\|\sigma\|_{H^1(u)}\leq \liminf_n \|\sigma\|_{H^1(u_n)}.
\]\label{semicont}
\end{prop}
There are several technical reasons why we use $M|\nabla \sigma|$ in the definition of the norm. Note however that the intuitive definition with just $\nabla \sigma$ would most certainly be too weak as $u$ could for instance vanish just at the points where $\nabla\sigma$ is very large. In particular Prop. \ref{semicont} would not be true.

We also need some similar $W^{1,1}$ assumptions on $F$. Following the definition of
$H^1(u)$, a first attempt would be
\begin{equation}
\|F\|_{W^{1,1}(u)}^2=\int_0^T\int_{\R^d}  M|\nabla F|(t,x)\,u(t,dx)\,dt.
\label{W110}
\end{equation}
Unfortunately while this definition would work, it is slightly too strong in some cases. This is due to the fact that the maximal
operator $M$ is bounded on $L^p$, $p>1$, but not on $L^1$. In particular if $u\in L^\infty$ then the norm defined in \eqref{H1} would
automatically be finite if $\sigma$ is in the usual $H^1$ space but the norm defined in \eqref{W110} would {\em not} be finite if $F\in
W^{1,1}$ in general.

Therefore in order to obtain better assumptions we have to work with a more complicated space. Define the modified maximal operator
\[
M_L f(x)=\sqrt{\log L}+\int_{B(x,1)} \frac{|f(z)|\,\mathbbm{1}_{|f(z)|\geq \sqrt{\log L}}\,dz}{(L^{-1}+|x-z|)\,|x-z|^{d-1}}.
\]
Now for any increasing $\phi$ with $\phi(\xi)/\xi\rightarrow \infty$ as $\phi\rightarrow \infty$, denote
\begin{equation}\begin{split}
\|F\|_{W^{\phi,\text{weak}}(u)}=&\sup_{L\geq 1} \frac{\phi(L)}{L\,\log L}\int_0^T \,\int_{\R^d}  (|F(t,x)|+M_L \nabla F(t,x))
\,u(t,dx)\,dt.\\
\end{split}\label{W11}
\end{equation}
As before, it is easy to show that this defines a well behaved norm
\begin{prop} The definition \eqref{W11} is a norm for any super linear $\phi$. Moreover if $F_n\longrightarrow F$ in the sense of distribution then
\[
\|F\|_{W^{\phi,weak}(u)}\leq \liminf_n \|F_n\|_{W^{\phi,weak}(u)},
\]
And if for a given $F$, $u_n\geq 0$ converges to $u$ for the weak-* topology of measures then
\[
\|F\|_{W^{\phi,weak}(u)}\leq \liminf_n \|F|_{W^{\phi,weak}(u_n)}.
\]\label{semicont2}
\end{prop}

In the one dimensional case, we can prove strong existence and pathwise uniqueness using $H^{1/2}$ type of assumptions on $\sigma$:
we define first as usual
\[
\partial^{1/2}_x\sigma={\cal F}^{-1} (|\xi|^{1/2}\,{\cal F} \sigma),
\]
with ${\cal F}$ the Fourier transform in $x$. Then we pose as before
\begin{equation}
\|\sigma\|_{H^{1/2}(u)}^2=\int_0^T\int_{\R^d}  (M|\partial^{1/2}_x\sigma|(t,x))^2\,u(t,dx)\,dt.
\label{H1/2}
\end{equation}
Again one has
\begin{prop} The definition \eqref{H1/2} is a norm. Moreover if $\sigma_n\longrightarrow \sigma$ in the sense of distribution then
\[
\|\sigma\|_{H^{1/2}(u)}\leq \liminf_n \|\sigma_n\|_{H^{1/2}(u)},
\]
And if for a given $\sigma$, $u_n\geq 0$ converges to $u$ for the weak-* topology of measures then
\[
\|\sigma\|_{H^{1/2}(u)}\leq \liminf_n \|\sigma\|_{H^{1/2}(u_n)}.
\]\label{semicont3}
\end{prop}

\subsection{General results on strong solutions to~\eqref{sde}}
\label{sec:gal-res}


In the multi-dimensional case, our most general result is the following one, proved in Section~\ref{proofmaintheo}.
\begin{thm} One has
  \begin{description}
  \item[\textmd{(i)}] Existence: Assume that there exists a sequence of smooth $F_n,\;\sigma_n\in L^\infty$, such that the solution $u_n$ to
    \eqref{pden} satisfies
    for some super linear $\phi$, {\em i.e.} $\phi(\xi)/\xi\rightarrow \infty$ as $\xi\rightarrow \infty$
    \begin{align}
      &\int_0^T\int_{\R^d} (|\sigma_n-\sigma|+|F_n-F|)\,du_n\,dt\longrightarrow
      0, \label{eq:hyp-1} \\
      &\sup_n
   \left(\|F\|_{W^{\phi,weak}(u_n)}+\|\sigma\|_{H^1(u_n)}+\|F_n\|_{L^\infty}+\|\sigma_n\|_{L^\infty}\right)
      <\infty, \label{eq:hyp-2} \\
      & u_n\longrightarrow u\ \mbox{for the weak-* topology of measures.} \label{eq:hyp-3} 
    \end{align}
     Then there exists a strong solution $X_t$ to \eqref{sde} s.t.
    $(X^n_t-\xi,t\in[0,T])_n$ converges in $L^p(\Omega,L^\infty([0,T]))$ for all $p>1$ to
    $(X_t-\xi,t\in[0,T])$, with $X_t^n$ the solutions to \eqref{sden}. In addition, $u(dt,dx)=u(t,dx)dt$, where $u(t,\cdot)$ is the
    law of $X_t$ for all $t\in[0,T]$.
  \item[\textmd{(ii)}] Uniqueness: Let $X$ and $Y$ be two solutions to \eqref{sde} with one-dimensional time marginals $u_X(t,\cdot)$
    and $u_Y(t,\cdot)$ on $[0,T]$. Assume that $F\in L^\infty$, $X_0=Y_0$ a.s. and that
    $$ \|F\|_{W^{\phi,weak}(u_X)}+\|F\|_{W^{\phi,weak}(u_Y)}+\|\sigma\|_{H^1(u_X)}+\|\sigma\|_{H^1(u_Y)}<\infty
    $$
    for some super linear function $\phi$. Then one has pathwise uniqueness: $sup_{t\in[0,T]} |X_t-Y_t|=0\ a.s.$
  \end{description}
  \label{maintheo}
\end{thm}
Note that we do not require any ellipticity on $\sigma$ for this result. In that sense we
cannot hope to have any smoothing effect from the Wiener process and the assumption on $F$
must be enough to provide well posedness in the purely deterministic setting ($\sigma=0$). In
this case, taking any $u_0\in L^\infty$, our result gives that there exists a unique solution
of $\dot{X}_t=F(t,X_t)$ with $X_0=\xi$ and with law $u\in L^\infty$ provided that there exists
a sequence of regularized $F_n$ s.t. $u_n\rightarrow u$ for the weak-$*$ topology with $u\in
L^\infty$ and a super linear $\phi$ s.t.
$$
\sup_{L\geq 1}\frac{\phi(L)}{L\log L}\|F+M_L \nabla F\|_{L^1([0,T]\times \RR^d)}.
$$
The first point is for example implied by the assumption $\text{div} F\in L^\infty$ and the
second one can be proved to hold if $F\in L^{1}_{t,loc}(W^{1,1}_x)$ as in the proof of
Corollary~\ref{maincorollary} below. Hence, we recover the classical results of DiPerna and
Lions~\cite{DL} but not the optimal $BV$ assumption from Ambrosio~\cite{Am}. 

In dimension $1$, the result is even better: we recover the $H^{1/2}$ type of assumption from \cite{YW71,LG,ES1}, but we lose a
little bit on $F$ (we have to use \eqref{W110} instead of \eqref{W11}).
\begin{thm} Assume that $d=1$. One has
  \begin{description}
  \item[\textmd{(i)}] Existence: Assume that there exists a sequence of smooth $F_n,\sigma_n\in L^\infty$, such that the solution
    $u_n$ to \eqref{pden}
    satisfies
    \[\begin{split}
      &\int_0^T\int_{\R} (|\sigma_n-\sigma|+|F_n-F|)\,du_n\,dt\longrightarrow 0,\\
      & \sup_n(\|F\|_{W^{1,1}(u_n)}+ \|\sigma\|_{H^{1/2}(u_n)}+\|F_n\|_{L^\infty}+\|\sigma_n\|_{L^\infty})<\infty,\\
      & u_n\longrightarrow u\ \mbox{for the weak-* topology of measures.}\\
    \end{split}\] 
    Then there exists a strong solution $X_t$ to \eqref{sde} s.t. $(X^n_t-\xi,t\in[0,T])_n$ converges in
    $L^p(\Omega,L^\infty([0,T]))$ for all $p>1$ to $(X_t-\xi,t\in[0,T])$, with $X_t^n$ the solutions to \eqref{sden}. In addition,
    $u(dt,dx)=u(t,dx)dt$, where $u(t,\cdot)$ is the law of $X_t$ for all $t\in[0,T]$.
  \item[\textmd{(ii)}] Uniqueness: Let $X$ and $Y$ be two solutions to \eqref{sde} with one-dimensional time marginals $u_X(t,\cdot)$
    and $u_Y(t,\cdot)$ on $[0,T]$. Assume that $F\in L^\infty$, $X_0=Y_0$ a.s. and that
    $$ \|F\|_{W^{1,1}(u_X)}+\|F\|_{W^{1,1}(u_Y)}+\|\sigma\|_{H^{1/2}(u_X)}+\|\sigma\|_{H^{1/2}(u_Y)}<\infty.
    $$
    Then pathwise uniqueness holds:  $sup_{t\in [0,T]} |X_t-Y_t|=0\ a.s.$ 
  \end{description}
  \label{theo1d}
\end{thm}

Of course, while precise, the norms given by \eqref{H1}--(\ref{W11}) or \eqref{H1/2}--(\ref{W110}) are not so simple to use. However
it is quite easy to deduce more intuitive results with the more usual $W^{1,p}$ norms. We recall that $M$ is continuous onto every
$L^p$ space for $1<p\leq \infty$ and hence appropriate Sobolev norms are controlled by the norms $\|\cdot\|_{H^1(u)}$ and
$\|\cdot\|_{W^{1,1}(u)}$ if some $L^{q}$ estimate is available on the law $u$.
 
One complication occurs when $u_{X}\in L^\infty$ and one wants to obtain the close to optimal $W^{1,1}$ assumption on $F$ (instead of
$W^{1,p}$ for some $p>1$) as the maximal function is not bounded onto $L^1$. This is the reason why we defined \eqref{W11}, which can
be used following \cite{Ja} (we recall the main steps in the appendix).

Therefore, Theorems~\ref{maincorollary} and~\ref{maincorollary-1d} are simple corollaries of Theorems~\ref{maintheo}
and~\ref{theo1d}, respectively, except for the previous complication for Theorem~\ref{maincorollary}.

In order to apply Theorems~\ref{maincorollary} and~\ref{maincorollary-1d}, we need to consider
cases where it is possible to obtain better integrability than $L^1$ bounds for a solution to
\eqref{pde}. This occurs in various situations, some of which will be studied in the next
Subsection. One difficulty to apply Theorems~\ref{maincorollary}~(ii)
and~\ref{maincorollary-1d}~(ii) is to obtain pathwise uniqueness without restriction on the
set of solutions considered. This will of course be ensured if uniqueness in law is known
for~(\ref{sde}). More precisely, if the conclusion of Theorem~\ref{maincorollary}~(i) or
Theorem~\ref{maincorollary-1d}~(i) holds and there is uniqueness in law for~(\ref{sde}), then
$u_X=u_Y=u$ for all solutions $X$ and $Y$ to~(\ref{sde}) as in
Theorem~\ref{maincorollary}~(ii) or Theorem~\ref{maincorollary-1d}~(ii) and hence pathwise
uniqueness holds. This argument will be used repeatedly in the next subsection.
%
\subsection{Consequences}
\label{sec:csq}
Let us first consider the case where $\sigma$ is uniformly
elliptic: for all $t,x$,
\begin{equation}
  \frac{1}{2}\sigma\,\sigma^*(t,x)=a(t,x)\geq c\,I\label{ellipticity}
\end{equation}
for some $c>0$.
For example if $F=0$ and $\sigma$ does not depend on time, then there exists a corresponding stationary measure $\bar u>0$ in
$L^{d/(d-1)}$ as per Aleksandrov~\cite{A}. In that case, when $u_0\leq C\bar u$, then the unique solution $u$ of~\eqref{pde} in
$L^2_{t,loc}(H^1_x)$ satisfies $u(t,dx)\leq C\bar u(x)dx$ for all $t\geq 0$ by the maximum principle.

\begin{cor}
  \label{cor:1d}
  Assume that $F=0$ and $\sigma(x)$ satisfies \eqref{ellipticity} and belongs to $L^\infty\cap
  W^{1,2d}_x$ (or $L^\infty\cap H^{1/2}$ if $d=1$).
  Assume also that $u_0\leq C\bar u$ for some constant $C>0$. Then one has both existence of a strong solution to \eqref{sde} and
  pathwise uniqueness. 
\end{cor}
Note that pathwise uniqueness holds without additional assumption since $\sigma\in W^{1,2d}$
implies that $\sigma$ is continuous. And uniqueness in law holds in this case since $\sigma$
and $F$ are bounded and $\sigma$ is uniformly elliptic~\cite[Thm.~7.2.1]{SV79}.

Those results were later extended by Krylov in the parabolic, time dependent
case~\cite{K71,K80}. We may for example use the following version found in~\cite{Zh11}.
\begin{thm}
  \label{thm:Krylov}
  Assume that $F$ and $\sigma$ are bounded and $\sigma$ satisfies~\eqref{ellipticity}. Then,
  for all solution $X$ of~\eqref{sde} with any initial condition, for all $T>0$ and $p,q>1$
  such that
  $$
  \frac{d}{p}+\frac{2}{q}<2,
  $$
  there exists a constant $C$ such that for all $f\in L^q_t(L^p_x)$
  $$
  \mathbb{E}\left[\int_0^T f(t,X_t)dt\right]\leq C\|f\|_{L^q_t(L^p_x)}.
  $$
\end{thm}
This result means that
$$
u\in L^{q'}_t(L^{p'}_x),
$$
where $1/p+1/p'=1$ and $1/q+1/q'=1$, and we obtain the following corollary.
\begin{cor}
  \label{cor:Krylov}
  \begin{description}
  \item[\textmd{(i)}] Assume that $d\geq 2$, $F,\;\sigma\in L^\infty$, $\sigma$ satisfies~\eqref{ellipticity}, $F\in
    L^{q/2}_{t,loc}(W^{1,p/2}_x)$ and $\sigma\in L^q_{t,loc}(W^{1,p}_x)$ with
    $2/q+d/p<1$. Then one has both existence of a strong solution to \eqref{sde} and pathwise uniqueness for any initial condition
    $\xi$.
  \item[\textmd{(ii)}] Assume that $d=1$, $F,\;\sigma\in
    L^\infty$, $\sigma$ satisfies~\eqref{ellipticity}, $\sigma\in L^q_{t,loc}(W^{1/2,p}_x)$
    with $2/q+1/p<1$ and  $F\in L^{q/2}_{t,loc}(W^{1,p/2}_x)$ if $p>2$, $F\in L^{q/2}_{t,loc}(W^{1,1+\varepsilon})$ for some
    $\varepsilon>0$ if $p\leq 2$. Then one has both existence of a strong solution to \eqref{sde} and pathwise uniqueness for any initial
    condition $\xi$.
  \end{description}
\end{cor}

Note that in this case, pathwise uniqueness holds without additional assumption since Krylov's
inequality implies that $u\in L^{q'}_t(L^{p'}_x)$ for \emph{all} solutions to~(\ref{sde}).

In our setting, since we need additional regularity on $\sigma$, it is easy to obtain better a
priori estimates for $u$ than those given by Krylov inequality. For instance:
\begin{prop} 
  \label{apriorilp}
 For any $d\geq 1$, assume $u^0\in L^1\cap L^\infty$, $F,\;\sigma\in L^\infty$, $\sigma$ satisfies~\eqref{ellipticity} and $\nabla
  \sigma\in L^q_{t,loc}(L^p_x)$ satisfying $2/q+d/p=1$ with $p>d$. Then any $u$ solution to \eqref{pde}, limit of smooth solutions, belongs to $L^\infty_t(L^r_x)$ for any $1\leq r\leq\infty$.
\end{prop}
This proposition is based on classical energy estimates and hence we just give a very short proof of it in
Section~\ref{sec:pf-prop-apriori}. Combined with Theorem \ref{maincorollary} this gives slightly better conditions for $\sigma$ and
much better conditions for $F$, assuming additional conditions on the initial distribution:
\begin{cor}
  \label{cor:better}
  Assume that $d\geq 2$, $u^0\in L^1\cap L^\infty$, $F,\;\sigma\in L^\infty$, $F\in L^1_{t,loc}(W^{1,1}_x)$ and $\nabla\sigma\in
  L^q_{t,loc}(L^p_x)$, where $2/q+d/p=1$ with $p>d$. Assume as well that $\sigma$ satisfies \eqref{ellipticity}. Then one has
  existence of a strong solution to \eqref{sde} with marginal distributions $u(t,dx)$ in $L^\infty_{t,loc}(L^\infty_x)$. In addition,
  pathwise uniqueness holds among all solutions with marginal distributions in $L^\infty_{t,loc}(L^\infty_x)$.
\end{cor}
As above, the pathwise uniqueness property could be improved if we could prove that uniqueness
in law holds. If $d=2$, uniqueness in law holds when $\sigma$ and $F$ are bounded and $\sigma$
is uniformly elliptic~\cite{K74}. When $d\geq 3$, by Sobolev embedding, the assumption
$\nabla\sigma\in L^q_{t,loc}(L^p_x)$ implies that $x\mapsto\sigma(t,x)$ is continuous for
almost all $t\geq 0$. This condition is not exactly sufficient to use the result of Stroock
and Varadhan~\cite[Thm.~7.2.1]{SV79}, which assumes that
$\sup_{t\in[0,T]}|\sigma(t,x)-\sigma(t,y)|\rightarrow 0$ when $y\rightarrow x$. This is true for
example if $\nabla\sigma\in L^\infty_{t,loc}(L^p_x)$ for $p>d$. 

Hence we obtain for \eqref{sde}
\begin{cor}
  \label{cor:better-v2}
  Assume that $d\geq 2$, $u^0\in L^1\cap L^\infty$, $F,\;\sigma\in L^\infty$, $F\in
  L^1_{t,loc}(W^{1,1}_x)$ and $\nabla\sigma\in L^q_{t,loc}(L^p_x)$ where $2/q+d/p=1$ with
  $p>d$. Assume as well that $\sigma$ satisfies \eqref{ellipticity}, and if
  $d\geq 3$ that for all $x$,
  $$
  \sup_{t\in[0,T]}|\sigma(t,x)-\sigma(t,y)|\rightarrow 0\quad\text{when}\quad y\rightarrow x.
  $$
  Then one has both existence of a strong solution to \eqref{sde} and pathwise uniqueness.
\end{cor}


This result can be compared with previous works dealing with the uniformly elliptic case. The best results in this case seem to be
those of~\cite{Zh11} and~\cite{PLL12}. In the first work, strong existence and pathwise uniqueness are proved under the assumptions
$\nabla\sigma\in L_{t,loc}^q(L_x^p)$, $\sigma(t,x)$ uniformly continuous with respect to $x$ and $F\in L_{t,loc}^q(L_x^p)$ with
$d/p+2/q<1$, so we obtain a slightly better condition on $\sigma$ (we can handle the limit
case $d/p+2/q=1$ and no uniform continuity is needed for strong existence), and a condition on $F$ which is neither stronger nor weaker, since
$L^1_{t,loc}(W^{1,1}_x)$ neither contains nor is contained in $L^q_{t,loc}(L^p_x)$ with $d/p+2/q<1$. In the second work, since the
approach for pathwise uniqueness is very different, the conditions obtained are of a different nature as ours. In particular, this
work requires additional boundedness assumptions on $\text{div}\,\sigma$ and $(D\sigma)^2$.
 
\medskip

In dimension $1$ in the stationary case, even if \eqref{ellipticity} is not satisfied but instead only 
\begin{equation}
\frac{1}{2}\sigma^2(x)=a(x)>0,\label{ellipticity'}
\end{equation}
then one has the a priori bound 
\[
u(t,x)\leq \frac{C}{a(x)}\,e^{\int_0^x \frac{F(y)}{a(y)}\,dy},
\] 
for solutions to \eqref{pde} again provided that $u^0$ satisfies the same bound.  Therefore, we obtain
\begin{cor}
  \label{cor:1d-better}
  Assume $d=1$, $\sigma,F\in L^\infty$, $\sigma$ satisfies~\eqref{ellipticity'}, $F/a\in L^1$,
  $$
  u_0(x)\leq \frac{C}{a(x)}\,e^{\int_0^x \frac{F(y)}{a(y)}\,dy}
  $$
  and
  \begin{equation}
    \label{eq:hyp-1d-better}
    \int_\R \frac{(M|\partial^{1/2}\sigma|(x))^2}{a(x)}\,dx<\infty\quad\text{and}\quad \int_\R\frac{M|\nabla F|}{a(x)}\,dx<\infty.    
  \end{equation}
  Then one has both existence of a strong solution to \eqref{sde} and pathwise uniqueness.
\end{cor}
Note that the assumptions~\eqref{eq:hyp-1d-better} imply that $a^{-1}\in L^1_{loc}$, which is a necessary and sufficient condition
for uniqueness in law when $F$ is bounded~\cite{ES1}.

We will prove in Lemma~\ref{maximal1/2} of Section~\ref{sec:technical} that for all $x,\;y$
\begin{equation}
  \label{eq:H1/2-ineq}
  |\sigma(x)-\sigma(y)|\leq \left(M|\partial^{1/2}_x \sigma|(x)+M|\partial^{1/2}_x \sigma|(y)\right)\,|x-y|^{1/2}.  
\end{equation}
This inequality allows us to compare our result with similar results of the literature~\cite{YW71,Zv74,LG,ES1}. The best conditions
in the time homogeneous case seem to be those of~\cite[Thm.~5.53]{ES1} and~\cite{YW71}. The first work assumes that
$|\sigma(x)-\sigma(y)|\leq C\sqrt{|x-y|}$ and that $F/a\in L^1_{loc}$, so our result gives better conditions for $\sigma$, but worse
conditions on $F$. The second work considers the time-inhomogeneous case and assumes that $F(t,x)$ is uniformly Lipschitz in the
variable $x$ and $|\sigma(t,x)-\sigma(t,y)|\leq h(|x-y|)$ where the function $h$ is strictly increasing and satisfies
$\int_{0^+}\rho^{-2}(u)du=+\infty$. Our conditions on $F$ are better, our conditions on $\sigma$ do not directly compare with the
last ones. However a typical scale for $h$ is  $h(u)=\sqrt{u}$ and in that case, by~(\ref{eq:H1/2-ineq}), our bound on $|\sigma(x)-\sigma(y)|$ is less
stringent. We can also compare with known results in the case where $F=0$. The best result in this case seems to be the one
of~\cite{ES1}, which assumes that there exist a function $f$ on $\RR$ and a strictly increasing function $h$ on $\RR_+$ such that
$|\sigma(y)-\sigma(x)|\leq f(x) h(|y-x|)$ for all $x,y$, $f/a\in L^1_{loc}$ and $\int_{0^+}\rho^{-2}(u)du=+\infty$. Our result is
slightly weaker than this one since we must take $h(u)=\sqrt{u}$ as above, but the bound~(\ref{eq:H1/2-ineq}) is more general than the
one of~\cite{ES1}.
\medskip


We point out that, in higher dimension as well, ellipticity is not always required for bounds on the law . We give the classical
example of SDE's in the phase space $\mathbb{R}^{2d}$
\begin{equation}
dX_t=V_t,\quad dV_t=F(t,X_t)dt+\sigma(t,X_t)\,dW_t,\quad X_0=x,\ V_0=v.\label{sdepspace}
\end{equation}
The law $u(t,x,v)$ of the joint process $(X_t,V_t)_{t\geq 0}$ solves the kinetic equation
\begin{equation}
\partial_t u(t,x,v)+v\cdot\nabla_x u(t,x,v)+F(t,x)\cdot\nabla_v u(t,x,v)=\sum_{1\leq i,j\leq d}a_{ij}(t,x)\frac{\partial^2
  u(t,x,v)}{\partial v_i\partial v_j}.\label{kinetic}
\end{equation}
Eq.~\eqref{kinetic} is in fact better behaved than \eqref{pde} for rough coefficients as its symplectic structure for instance
guarantees that it satisfies a maximum principle for all measure-valued solutions: note that the rough coefficients are only in $t,x$ and
always multiply derivatives in $v$. 
In particular there is uniqueness among
all measure-valued solutions, and if $u^0\in L^\infty(\R^{2d})$ then $u\in L^\infty(\R_+\times\R^{2d})$. This is true even though the
diffusion in \eqref{sdepspace} is degenerate (there is no diffusion in the $x$ direction, and $\sigma$ can also be degenerate). 

Hence in
this situation, one obtains an even better result.
\begin{cor}
  \label{cor:kinetic}
  Assume that $\sigma \in L^\infty\cap L^2_{t,loc}(H^{1}_x)$ and $F\in L^1_{t,loc}(W^{1,1}_x)$. Assume also that $u_0\in
  L^{\infty}$. Then one has both existence of a strong solution to \eqref{sdepspace} and
  pathwise uniqueness.
\end{cor}
 \medskip

To conclude, let us observe that most of the previous results give strong existence for smooth (and so non deterministic) initial
distributions. However, one can use the next result to obtain strong existence and pathwise uniqueness for almost all deterministic
initial conditions.
\begin{prop}
  \label{prop:ae-well-posedness}
  Under the assumptions of either Corollary~\ref{cor:1d}, Corollary~\ref{cor:better-v2}, Corollary~\ref{cor:1d-better} or
  Corollary~\ref{cor:kinetic}, for any complete filtered probability space $(\Omega,(\mathcal{F}_t)_{t\geq 0},\PP)$ equipped with a
  $r$-dimensional standard Brownian motion $W$, there is strong existence and pathwise uniqueness for~\eqref{sde} on
  $(\Omega,(\mathcal{F}_t)_{t\geq 0},\PP,W)$ for almost all deterministic initial condition
  $\xi=x\in\mathbb{R}^d$.
\end{prop}

The proofs of the previous results are organized as follows. We start in Section~\ref{sec:technical} with some simple technical
proofs, including those of Propositions~\ref{semicont},~\ref{semicont2} and~\ref{semicont3}, Section~\ref{proofmaintheo} is then
devoted to the proof of Theorem~\ref{maintheo}, Section~\ref{sec:pf-1d} to the proof of Theorem~\ref{theo1d},
Section~\ref{sec:pf-prop-apriori} to the proof of Proposition~\ref{apriorilp}, and Section~\ref{sec:pf-prop-ae-wp} to the proof of
Proposition~\ref{prop:ae-well-posedness}. The proof of Theorem~\ref{maincorollary} is given in Appendix~\ref{sec:pf-cor}.
\section{Useful technical results}
\label{sec:technical}
The results and proofs presented in this section are mostly easy extensions
of well-known techniques, which we need in following sections and hence include here for the sake of completeness.
\subsection{Pointwise difference estimates}
We often need to estimate the difference of the coefficients $\sigma$ of $F$ at two different points $x$ and $y$ during the proofs. Hence we collect here all the results which allow us to do so and that we later use.
In all those estimates, time is only a parameter and we accordingly omit the time variable in most formulas.
 
We start by recalling the classical
inequality (see \cite{St} for instance)
\begin{equation}
|\sigma(t,x)-\sigma(t,y)|\leq (M|\nabla_x \sigma|(t,x)+M|\nabla_x \sigma|(t,y))\,|x-y|. \label{maximal}
\end{equation} 
We next turn to an extension with the operator $M_L$ used in the definition \eqref{W11} 
\begin{lem} Assume that $F$ is measurable and $|\nabla F|$ is a measure then for any $x,\;y\in\R^d$
\[
|F(t,x)-F(t,y)|\leq (h(t,x)+h(t,y))\,\left(|x-y|+\frac{1}{L}\right),
\]
with $h(t,x)=|F(t,x)|+M_L\nabla F(t,x)$.
\label{maximalF}
\end{lem}

\begin{proof} First observe that by the definition of $h$, the result is obvious if $|x-y|\leq 1/L$ or if $|x-y|\geq 1$. Assume now that $1/L\leq |x-y|\leq 1$.
We recall the Lemma from \cite{Ja}
\begin{lem} Assume $F$ is measurable and $|\nabla F|$ is a measure. 
There exists a constant $C$ s.t. for any $x$, $y$
\begin{equation}
|F(x)-F(y)|\leq C\,\int_{B(x,y)}
|\nabla F(z)|\,\left(\frac{1}{|x-z|^{d-1}}+\frac{1}{|y-z|^{d-1}}
\right)\,dz,\label{diff}
\end{equation}
where $B(x,y)$ denotes the ball of center $(x+y)/2$ and diameter
$|x-y|$. \label{lemma}
\end{lem}
Now $|\nabla F|\leq \sqrt{\log L}+|\nabla F|\,\ind_{|\nabla F|\geq \sqrt{log L}}$ and thus
\[
\frac{1}{|x-y|}\int_{B(x,y)}
\frac{|\nabla F(z)|}{|x-z|^{d-1}}\,dz\leq C\,\sqrt{\log L}+\int_{B(x,1)}\frac{|\nabla F(z)\,\ind_{|\nabla F|\geq \sqrt{log L}}|}{\,(1/L+|x-z|)\,|x-z|^{d-1}},
\]
where we used that if $z\in B(x,y)$ then $|x-z|+1/L\leq 2\,|x-y|$.
By the definition of $M_L$, this concludes the proof.\end{proof}

\medskip

Let us turn now to our last bound which uses $\partial^{1/2}_x \sigma$
\begin{lem} Assume that $\sigma$ and $\partial^{1/2}_x \sigma$ are measurable then for any $x,\;y$ 
\[\begin{split}
|\sigma(t,x)-\sigma(t,y)|\leq \left(M|\partial^{1/2}_x \sigma|(t,x)+M|\partial^{1/2}_x \sigma|(t,y)\right)\,|x-y|^{1/2}.
\end{split}\]\label{maximal1/2}
\end{lem}
\begin{proof}
By the definition of $\partial^{1/2}_x\sigma$
\[
\sigma(x)=K\star \partial^{1/2}_x \sigma,
\]
for the convolution kernel $K$ with $\hat K=|\xi|^{-1/2}$, which implies that
\begin{equation}
|K(x)|\leq C\,|x|^{d-1/2},\qquad |\nabla K(x)|\leq C\,|x|^{d+1/2}.\label{boundK}
\end{equation}
Now simply compute
\[\begin{split}
|\sigma(x)-\sigma(y)|\leq &\int_{|z-x|\geq2\,|x-y|} |K(x-z)-K(y-z)|\,|\partial^{1/2}_x \sigma(z)|\,dz\\
&+\int_{|z-x|\leq2\,|x-y|} (|K(x-z)|+|K(y-z)|)\,|\partial^{1/2}_x \sigma(z)|\,dz.
\end{split}\]
Denote $|x-y|=r$. One has by \eqref{boundK}
\[\begin{split}
\int_{|z-x|\leq2\,r} |K(x-z)|\,&|\partial^{1/2}_x \sigma(z)|\,dz\leq \sum_{n\geq -1} \int_{|z-x|\leq 2^{-n}\,r} \frac{2^{n(d-1/2)}}{r^{d-1/2}} \,|\partial^{1/2}_x \sigma(z)|\,dz\\
&\leq \sum_{n\geq -1} 2^{-n/2} r^{1/2}\,M |\partial^{1/2}_x \sigma|(x)=C\,r^{1/2}\,M |\partial^{1/2}_x \sigma|(x).
\end{split}\]
Since $|z-x|\leq 2r$ implies that $|z-y|\leq 3r$, one has the same inequality
\[
\int_{|z-x|\leq2\,r} |K(y-z)|\,|\partial^{1/2}_x \sigma(z)|\,dz\leq C\,r^{1/2}\,M |\partial^{1/2}_x \sigma|(y).
\]
As for the last term, first note that if $|x-z|\geq2\,|x-y|$ then $|y-z|\geq |x-z|/2$. Hence by \eqref{boundK} if $|x-z|\geq2\,|x-y|$
\[
|K(x-z)-K(y-z)|\leq C\,\frac{|x-y|}{|x-z|^{d+1/2}}.
\]
Therefore
\[\begin{split}
\int_{|z-x|\geq2\,|x-y|} &|K(x-z)-K(y-z)|\,|\partial^{1/2}_x \sigma(z)|\,dz\\
&\leq \sum_{n\geq 1} \int_{|z-x|\geq 2^n\,r} C\,\frac{r}{(2^n r)^{d+1/2}}\,|\partial^{1/2}_x \sigma(z)|\,dz\\
&\leq C\,r^{1/2}\sum_{n\geq 1} 2^{-n/2}\,M |\partial^{1/2}_x \sigma|(x)\leq C\,r^{1/2}\,M |\partial^{1/2}_x \sigma|(x).
\end{split}\]
Summing up the three estimates concludes the proof.
\end{proof}
\subsection{Proof of Prop. \ref{semicont}, \ref{semicont2}, \ref{semicont3}}
Note that the time variable again plays essentially no role here.
We may of course assume that $u$ is not identically $0$. 

Let us first check that $\|.\|_{H^1(u)}$ is a norm. The triangle inequality and the equality $\|\lambda \sigma\|_{H^1(u)}=|\lambda|\,\|\sigma\|_{H^1(u)}$ are straightforward. Now assume that $\|\sigma\|_{H^1(u)}=0$.

This means that $\sigma=0$ and $M|\nabla \sigma|=0$ on the support of $u$. Since this support contains at least one point and by the definition of the maximal operator, $\nabla\sigma$ is identically $0$. Thus $\sigma$ is a constant which is necessarily $0$ as it has to vanish on the support of $u$. 

Now let $\sigma_n\longrightarrow \sigma$ in the sense of distributions. Take any point $x$ and any $c>1$
\[\begin{split}
\frac{1}{|B(0,r)|}\int_{B(0,r)} |\nabla\sigma|(x+z)\,dz&\leq \frac{1}{|B(0,r)|}\liminf \int_{B(0,cr)} |\nabla\sigma_n|(x+z)\,dz\\
&\leq c^d \liminf M|\nabla\sigma_n|(x). 
\end{split}\]
Taking now the supremum in $r$, we deduce that for any $c>1$
\[
M|\nabla \sigma|(x)\leq c^d \liminf M|\nabla\sigma_n|(x).
\]
Apply now Fatou's lemma and let $c$ go to $1$ to deduce
\[
\int (M|\nabla \sigma|(x))^2 u(dx)\leq \liminf \int (M|\nabla \sigma_n|(x))^2 u(dx).
\]
Using \eqref{maximal}, it is easy to deduce the same inequality on the full $H^1(u)$ norm. 

Let us now turn to the last part. Denote $f=|\nabla \sigma|$, $f$ is a non negative, measurable function, possibly with $+\infty$ values on large sets. Then $g=(Mf)^2$ is non negative, measurable, lower semi-continuous and again possibly with $+\infty$ values on large sets, see \cite{St} for instance. Note that for any positive measure $\mu$
\[
\int g\,d\mu=\int_0^\infty \int \ind_{g(x)>\xi} \mu(dx)\,d\xi.
\]
Now assume $u_n\rightarrow u$ in the weak-* topology of $M^1$ with $u_n\geq 0$. Note that for any open set $O$
\[
\int_O du\leq \liminf \int_O du_n.
\]
Take $O=\{g(x)>\xi\}$ which is open by the lower semi-continuity of $g$. Therefore
\[
\int g\,du\leq \liminf\int g\,du_n,
\]
which finishes the proof of Prop. \ref{semicont}.

Prop. \ref{semicont2} and \ref{semicont3} are proved in exactly the same manner.
\section{Proof of Theorem \ref{maintheo}\label{proofmaintheo}}
We use two types of estimates; one is based on an explicit quantitative estimate which
generalizes the one in \cite{CD} for Ordinary Differential Equations and one which generalizes
the local time which is used in dimension $1$ in the classical approach~\cite{YW71,LG,ES1}. We
use the first quantitative estimate to prove existence and the second one to prove uniqueness
(though with suitable modifications any one could be used for both existence and uniqueness).

The first method is more precise but slightly more complicated than the second.
\subsection{Existence}
We consider the sequence of solutions to the regularized problem \eqref{sden}, and assume it satisfies the assumptions of Th.~\ref{maintheo}.
%
We fix $T>0$ in all the proof. The proof is based on estimates on the expectation of the
family of quantities
\begin{equation}
Q^{(\varepsilon)}_{nm}(t)=\log \left(1+ \frac{|X_t^n-X_t^m|^2}{\eps^2}\right),
\label{Qnm}
\end{equation}
given in the following lemma.

\begin{lem}
  \label{lem:1}
  There exists a constant $C$ such that, for all $0<\varepsilon\leq 1$ and $n,m\geq 1$,
  \begin{equation}
    \sup_{t\in[0,\ T]} \EE(Q^{(\eps)}_{nm}(t))\leq C\,(1+|\log \eps|\,\tilde \eta(\eps))+C\,\frac{\eta(n,m)}{\eps^2},\label{controlQ}
  \end{equation}
  where $\eta(n,m)\rightarrow 0$ when $n,m\rightarrow+\infty$ and $\tilde
  \eta(\eps):=(\varepsilon\phi(\varepsilon^{-1}))^{-1}\rightarrow 0$ when
  $\varepsilon\rightarrow 0$.
\end{lem}

\begin{proof}
Note that
$$
|\nabla (\log(1+|x|^2/\eps))|=\left|\frac{2x}{\eps^2+|x|^2}\right|\leq\frac{C}{\varepsilon+|x|}
$$
and
\[\begin{split}
|\nabla^2 (\log(1+|x|^2/\eps))|&=\left|\nabla\left(\frac{2x}{\eps^2+|x|^2}\right)\right|
\leq \frac{C}{\eps^2+|x|^2}. 
\end{split}\] 
Thus, by It\^o's formula and since
$\sup_n(\|\sigma_n\|_\infty+\|F_n\|_\infty)<+\infty$, for any $C^2_b$ function $f$,
\begin{align}
  \EE(f(X_t^n-X_t^m))=& f(0)+\frac{1}{2}\int_0^t \EE\big( \nabla^2
  f(X_s^n-X_s^m)\,(\sigma_n\sigma^*_n(X_s^n) \notag \\
  &+\sigma_m\sigma^*_m(X_s^m)-\sigma_n(X_s^n)
  \sigma_m(X_s^m)-\sigma_m(X_s^m)\sigma^*_n(X_s^n)\big)\,ds \notag \\
  &+\int_0^t \EE(\nabla f(X_s^n-X_s^m)\cdot ( F(s,X_s^n)- F(s,X_s^m)))\,ds \notag \\
  \leq & f(0)+\frac{1}{2}\int_0^t \EE\Big(|\nabla^2
  f(X_s^n-X_s^m)\,\big(|\sigma(X_s^n)-\sigma(X_s^m)|^2 \notag \\
  +\sup_k\|\sigma_k\|_{L^\infty}\,&(|\sigma_n(X_s^n)-\sigma(X_s^n)|+
  |\sigma_m(X_s^m)-\sigma(X_s^m)|)\big)\Big)ds \notag \\
  &+\int_0^t \EE(|\nabla f(X_s^n-X_s^m)|\,|F_n(s,X_s^n)-F_m(s,X_s^m)|)\,ds.
  \label{eq:calcul-1}
\end{align}
Hence
\begin{align}
  \EE(Q^{(\varepsilon)}_{nm}(t))\leq & C\int_0^t \EE\left(\frac{|\sigma(s,X_s^n)-\sigma(s,X_s^m)|^2}{\eps^2+|X_s^n-X_s^m|^2}\right)\,ds+
  C\,\frac{\eta(n,m)}{\eps^2} \notag \\
  &+C\int_0^t \EE\left(\frac{|F(s,X_s^n)- F(s,X_s^m)|}{\eps+|X_t^n-X_t^m|}\right)\,ds,
  \label{eq:calcul-2}  
\end{align}
with $C$ a constant independent of $n$ and $\eps$ and $\eta(n,m)\rightarrow 0$ as
$n,\,m\rightarrow \infty$ by Assumptions~(\ref{eq:hyp-1}) and~(\ref{eq:hyp-2}).

Now as $\|\sigma\|_{H^1(u_n)}+\|\sigma\|_{H^1(u_m)}<\infty$, denoting $h=M|\nabla\sigma|$, 
\[
\int_0^T\int h^2(t,x)\,(u_n(t,dx)+u_m(t,dx))\,dt\leq \|\sigma\|_{H^1(u_n)}+\|\sigma\|_{H^1(u_m)}\leq C,
\]
with $C$ independent of $n$, $m$ and $\eps$. Now,
\[
\int_0^t \EE\left(\frac{|\sigma(s,X_s^n)-\sigma(s,X_s^m)|^2}{\eps^2+|X_s^n-X_s^m|^2}\right)\,ds\leq \int_0^t \EE(h^2(s,X_s^n)+h^2(s,X_s^m))\,ds.
\]
As $u_n$ is the law of $X_s^n$ then one obtains that
\[
\int_0^t \EE\left(\frac{|\sigma(s,X_s^n)-\sigma(s,X_s^m)|^2}{\eps^2+|X_s^n-X_s^m|^2}\right)\,ds\leq C.
\]
We now turn to the term involving $F$ and introduce the corresponding $h=|F|+M_{1/\eps}\nabla F$. 

By Lemma \ref{maximalF}
\[\begin{split}
\int_0^t \EE&\left(\frac{|F(s,X_n^s)-F(s,X_m^s)|}{\eps+|X_n^s-X_m^s|}\right)\,ds
\\
&\leq \int_0^t\int h(s,x)\,(u_n(s,x)+u_m(s,x))\,dx\,ds\\
\end{split}\]
By the definition \eqref{W11}
\[\begin{split}
 \int_0^t\int & h(s,x)\,(u_n(s,x)+u_m(s,x))\,dx\,ds\leq C\,\frac{|\log \eps|}{\eps\,\phi(\eps^{-1})}.\\
\end{split}\]
Define $\tilde \eta(\eps)=(\eps\,\phi(\eps^{-1}))^{-1}\rightarrow 0$ as $\eps\rightarrow 0$ since $\phi$ is super linear.

Combining the previous inequalities, we obtain~\eqref{controlQ}.
\end{proof}

Fix $p>1$. The next step consists in deducing from Lemma~\ref{lem:1} that $(X_t^n-\xi)$ is a
Cauchy sequence in $L^p(\Omega,L^\infty([0,T]))$. Since $F_n$ and $\sigma_n$ are uniformly
bounded, it is standard to prove that $X^n_t-\xi\in L^p(\Omega,\ L^\infty([0,\ T]))$ for all $n\geq 1$, so we only need to
prove the next lemma.
\begin{lem}
  \label{lem:2}
  For all $p>1$,
  \begin{equation}
    \EE\left(\sup_{t\in [0,\ T]}|X_t^n-X_t^m|^p\right)\longrightarrow 0\ \quad \mbox{as}\ n,\,m\rightarrow +\infty.\label{Cauchy}
  \end{equation}
\end{lem}

\begin{proof}
For fixed $t$, for any $\eps$ and $L$ s.t.
$0<\varepsilon<L$,
\[\begin{split}
\EE(|X_t^n-X_t^m|^p)\leq &\EE(|X_t^n-X_t^m|^p; |X_t^n-X_t^m|\geq L)+\eps^{p/2}\\
&+L^p \PP(|X_t^n-X_t^m|\geq \sqrt{\eps}).
\end{split}\]
Note that
\[
\EE(|X_t^n-X_t^m|^p; |X_t^n-X_t^m|\geq L)\leq
\frac{1}{L}(\EE(|X_t^n-\xi|^{p+1})+\EE(|X_t^m-\xi|^{p+1})).
\]
By the inequality of Burkholder-Davis-Gundy and since $F_n$ and $\sigma_n$ are uniformly
bounded, it is standard to check that
$$
\sup_{n\geq 1,t\in[0,T]}\EE(|X_t^n-\xi|^{p+1})<+\infty.
$$
Finally
\[
\PP(|X_t^n-X_t^m|\geq \sqrt{\eps})\leq \frac{\EE Q^{(\eps)}_{nm}(t)}{|\log \eps|}.
\]
Thus
\[
\EE(|X_t^n-X_t^m|^p)\leq C\left[\frac{1}{L}+\eps^{p/2}+\frac{L^p}{|\log \eps|}\left(1+|\log \eps|\,\tilde \eta(\eps)+\frac{\eta(n,m)}{\eps^2}\right)\right].
\]
Taking for example $\varepsilon^2=\eta(n,m)$ and $L=\left(\frac{1}{|\log
    \varepsilon|}+\tilde{\eta}(\varepsilon)\right)^{-1/2p}$, one concludes that
$$
\sup_{t\in[0,T]}\EE(|X^n_t-X^m_t|^p)\rightarrow 0\quad\text{\ as\ }n,m\rightarrow+\infty
$$
holds.

In order to pass the supremum inside the expectation, it suffices to observe that the
computation of~(\ref{eq:calcul-1}--\ref{eq:calcul-2}) in the proof of Lemma~\ref{lem:1} can be
applied to
$|A^n_{t\wedge\tau}-A^m_{t\wedge\tau}|^2\vee|M^n_{t\wedge\tau}-M^m_{t\wedge\tau}|^2$, where
$\tau$ is any stopping time and $X^n_t=\xi+A^n_t+M^n_t$ is Doob's decomposition of the
semi martingale $X^n_t$, i.e.\
$$
A^n_t=\int_0^t F(s,X^n_s)ds\qquad\text{and}\qquad M^n_t=\int_0^t \sigma(s,X^n_s)dW_t.
$$
Note that to be fully rigorous, one first needs to regularize the supremum $\vee$. 

Instead
of~(\ref{eq:calcul-2}), we obtain
\begin{multline*}
\EE\log\left(1+\frac{|A^n_{t\wedge\tau}-A^m_{t\wedge\tau}|^2\vee|M^n_{t\wedge\tau}-M^m_{t\wedge\tau}|^2}{\varepsilon^2}\right) \\
   \begin{aligned}
     \leq & C\int_0^t \EE\left(\frac{|\sigma(s,X_s^n)-\sigma(s,X_s^m)|^2}
       {\eps^2+|A^n_t-A^m_t|^2\vee|M^n_t-M^m_t|^2}\right)\,ds+
     C\,\frac{\eta(n,m)}{\eps^2} \\
     &+C\int_0^t \EE\left(\frac{|F(s,X_s^n)-
         F(s,X_s^m)|}{\eps+|A^n_t-A^m_t|\vee|M^n_t-M^m_t|}\right)\,ds, \\ 
\end{aligned}
\end{multline*}
or
\begin{multline*}
\EE\log\left(1+\frac{|A^n_{t\wedge\tau}-A^m_{t\wedge\tau}|^2\vee|M^n_{t\wedge\tau}-M^m_{t\wedge\tau}|^2}{\varepsilon^2}\right) \\
\begin{aligned}
\leq
     &C\int_0^t
\EE\left(\frac{|\sigma(s,X_s^n)-\sigma(s,X_s^m)|^2}{\eps^2+\frac{1}{4}|X_s^n-X_s^m|^2}\right)\,ds+ 
     C\,\frac{\eta(n,m)}{\eps^2} \\
     &+C\int_0^t \EE\left(\frac{|F(s,X_s^n)-
         F(s,X_s^m)|}{\eps+\frac{1}{2}|X_t^n-X_t^m|}\right)\,ds.
   \end{aligned}
\end{multline*}
Therefore, the same computation as in Lemma~\ref{lem:1} gives
$$
\sup_{t\in[0,T],\ \tau\text{\ stopping time}}\EE(|A^n_{t\wedge\tau}-A^m_{t\wedge\tau}|^p\vee|M^n_{t\wedge\tau}-M^m_{t\wedge\tau}|^p)\rightarrow 0\quad\text{\ as\ }n,m\rightarrow+\infty.
$$
Since $p>1$, Doob's inequality entails
$$
\EE(\sup_{t\in[0,T]}|M^n_{t}-M^m_{t}|^p)\rightarrow 0\quad\text{\ as\ }n,m\rightarrow+\infty.
$$
Fix $\eta>0$, and fix $n_0$ such that
$$
\sup_{t\in[0,T],\ \tau\text{\ stopping
    time}}\EE(|A^n_{t\wedge\tau}-A^m_{t\wedge\tau}|^p)\leq\eta
$$
for all $n, m\geq n_0$. For all $M>0$, let $\tau=\inf\{t\geq 0:|A^n_t-A^m_t|\geq M\}$. Then
$$
\PP(\sup_{t\in[0,T]}|A^n_t-A^m_t|\geq M)=\PP(\tau\leq T)\leq\frac{\eta}{M^p}.
$$
Now, for all $1<q<p$,
\begin{align*}
  \EE(\sup_{t\in[0,T]}|A^n_t-A^m_t|^q) & =q\int_0^{+\infty}
  x^{q-1}\PP(\sup_{t\in[0,T]}|A^n_t-A^m_t|\geq x)dx \\ & \leq q\int_0^{+\infty}
  x^{q-1}\left(\frac{\eta}{x^p}\wedge 1\right)dx=\frac{p\,\,\eta^{q/p}}{p-q}.
\end{align*}
Therefore
$$
\EE(\sup_{t\in[0,T]}|A^n_{t}-A^m_{t}|^q)\rightarrow 0\quad\text{\ as\ }n,m\rightarrow+\infty,
$$
which concludes the proof of \eqref{Cauchy}.
\end{proof}

From the fact that $(X^n-\xi)$ is a Cauchy sequence in $L^p(\Omega,L^\infty([0,T]))$, it is
standard to deduce the almost sure convergence for the $L^\infty$ norm of a subsequence of
$(X^n_t,t\in[0,T])_n$ to a process $(X_t,t\in[0,T])$ such that $(X_t-\xi,t\in[0,T])\in
L^p(\Omega,L^\infty([0,T]))$ for all $p>1$. Since the convergence holds for the $L^\infty$
norm, the process $X$ is a.s. continuous and adapted to the filtration $(\mathcal{F}_t)_{t\geq
  0}$. 

Since $u_n$ converges to $u$ in the weak-* topology of measures, we have for all bounded continuous function $f$ on
$[0,T]\times\RR^d$
$$
\mathbb{E}\int_0^T f(t,X_t)dt=\int_{\RR^d}\int_0^T f(t,x)u(dt,dx).
$$
so $u(dt,dx)=u(t,dx)dt$, where $u(t,dx)$ is the law of $X_t$.

Defining for all $t\in[0,T]$
\[
Y_t:=\int_0^t F(s,X_s)ds+\int_0^t\sigma(s,X_s)dW_s,
\]
it only remains to check that $Y_t=X_t$ for all $t\in[0,T]$, a.s., i.e.\ by continuity that $Y_t=X_t$ a.s., for all $t\in[0,T]$.

As 
\[
X_t^n=\int_0^t F_n(s,X_s^n)ds+\int_0^t\sigma_n(s,X_s^n)\,dW_s,
\]
one has $Y_t=X_t$ provided that
\[
\int_0^t \EE(|F_n(s,X_s^n)-F(s,X_s)|+|\sigma_n(s,X_s^n)-\sigma(x,X_s)|^2)\,ds\longrightarrow 0.
\]
From the assumption \eqref{eq:hyp-1} and the $L^\infty$ bounds on $F$ and $\sigma$ \eqref{eq:hyp-2}, this is implied by: For any fixed $\eps$ 
\[
\int_0^T \PP(|F(s,X_s^n)-F(s,X_s)|>\eps)+\PP(|\sigma(s,X_s^n)-\sigma(x,X_s)|>\eps)\,ds\longrightarrow 0.
\]
By the almost sure convergence of $X_s^n$ this would be automatic if $F$ and $\sigma$ are continuous or if the law $u_n$ was absolutely continuous with respect to the Lebesgue measure and equi-integrable (using then the Lebesgue points of $F$ and $\sigma$). In general however we require some additional work. We prove it for $\sigma$, the argument for $F$ being fully similar.

By Prop. \ref{semicont}
\[\begin{split}
\int_0^T \int_{\R^d} (M|\nabla \sigma(t,x)|)^2\,&(u(t,dx)+u_n(t,dx))\,dt\\
&\leq  \|\sigma\|_{H^1(u_n)}+\liminf\|\sigma\|_{H^1(u_n)}\leq C.
\end{split}\]
Now by \eqref{maximal}
\begin{multline*}
  \PP(|\sigma(s,X_s^n)-\sigma(s,X_s)|>\eps) \\
\leq \PP((M|\nabla\sigma|(s,X_s^n)+M|\nabla \sigma|
  (s,X_s))>\eps/|X_s^n-X_s|) \\ \leq\PP(|X^n_s-X_s|>\varepsilon^2)+\PP(M|\nabla\sigma|(s,X^n_s)\geq\frac{1}{2\varepsilon})+\PP(M|\nabla\sigma|(s,X_s)\geq\frac{1}{2\varepsilon}),
\end{multline*}
and one easily concludes as $|X_s^n-X_s|\longrightarrow 0$ almost surely.

Note that this shows that for this precise point, $\|\sigma\|_{H^s(u)}<\infty$ for some $s\in(0,1)$ would be enough
instead of $\|\sigma\|_{H^1(u)}<\infty$.

\subsection{Uniqueness}
Consider two solutions $X$ and $Y$ satisfying the assumptions of point~(ii) in Th.~\ref{maintheo}. Define a family of functions
$(L_\eps)_\eps$ in $C^\infty(\R^d)$ satisfying
\[
L_\eps(x)=
1\ \mbox{if}\ |x|\geq \eps,\quad L_\eps(x)=0\ \mbox{if}\ |x|\leq \eps/2,\quad \eps\,\|\nabla L_\eps\|_{L^\infty}+\eps^2\,\|\nabla^2 L_\eps\|\leq C,\\
\]
with $C$ independent of $\eps$, and $L_\eps(x)\geq L_{\eps'}(x)$ for all $\eps\leq\eps'$ and $x\in\RR^d$. Use It\^o's formula 
\[\begin{split}
\EE(L_\eps(X_t-Y_t))=&L(0)+\int_0^t \EE\left(\nabla L_\eps(X_s-Y_s)\cdot(F(s,X_s)-F(s,Y_s)\right)\,ds\\
&+\int_0^t \EE\Big(\nabla^2 L_\eps(X_s-Y_s):(\sigma\sigma^* (X_s)\\
&+\sigma\sigma^*(Y_s)-\sigma(X_s)\sigma^*(Y_s)-\sigma(Y_s)\sigma^*(X_s))
\big)\,ds.
\end{split}\] 
Hence 
\[\begin{split}
\EE(L_\eps(X_t-Y_t))\leq C\,\int_0^t \EE\Big(\mathbbm{1}_{\eps/2\leq|X_t-Y_t|\leq \eps}\,&\Big(
\frac{|\sigma(s,X_s)-\sigma(s,Y_s)|^2}{\eps^2}\\
&+\frac{|F(s,X_s)-F(s,Y_s)|}{\eps}
\Big)\Big)\,ds.
\end{split}\]
Now denote $h=M|\nabla\sigma|$ so that
\[
\int_0^T\int |h(t,x)|^2\,(u_X(t,dx)+u_Y(t,dx))\,dt\leq C<\infty.
\]
Define as well $\tilde h_\eps=|F|+M_{1/\eps} \nabla F$ s.t.
\[
\int_0^T \int \tilde h_\eps\,(u_X+u_Y)\,dx\,ds\leq \frac{C\,|\log \eps|}{\eps\,\phi(\eps^{-1})}.
\]
The corresponding computation involving $\tilde h_\eps$ is now tricky, precisely because of the dependence on $\eps$ in $\tilde h_\eps$. To simplify it, we will use a slightly different definition.

First note that one can always replace $\phi$ by a function growing slower (as long as it is
still super linear). Without loss of generality, we may hence assume that $\phi(\xi)/\xi$ is a
non-decreasing function which grows at most like $\log \xi$ and in particular that
\[
\frac{1}{C} \eps\, \phi(\eps^{-1})\leq \frac{\phi(\xi)}{\xi}\leq C\,\eps\, \phi(\eps^{-1})\,\quad\forall \xi\in [\eps^{-1/2},\ \eps^{-1}].
\]
Consider a partition of $(0,1)$ in $\bigcup_i I_i$ where the $I_i=[a_i,\;b_i)$ are
disjoint with $b_i=\sqrt{a_i}$ (except for $I_0:=[1/2,1)$) so that 
\[
|I_i|\sim\sqrt{a_i} \quad\text{when\ }i\rightarrow+\infty.
\] 
Now for any $\eps\in I_i$, choose $\bar h_\eps=\tilde h_{a_i}$. One has
\[
\int_0^T\int \bar h_\eps(t,x)\,(u_X(t,x)+u_Y(t,x))\,dx\,dt\leq C\,
\frac{|\log \eps|}{\eps\,\phi(\eps)}\leq 2C^2\,\frac{|\log b_i|}{b_i\,\phi(b_i^{-1})}.
\]
Now by \eqref{maximal} and Lemma \ref{maximalF}
\[\begin{split}
\EE(L_\eps(X_t-Y_t))\leq & C\,\int_0^t \EE\left[(h^2(s,X_s)+h^2(s,Y_s))\mathbbm{1}_{\eps/2\leq|X_t-Y_t|\leq \eps}\right]\,ds\\
&+C\,\int_0^t \EE\left[(\bar h_\eps(s,X_s)+\bar h_\eps(s,Y_s))\,\mathbbm{1}_{\eps/2\leq|X_t-Y_t|\leq \eps}\right]\,ds.
\end{split}\]
Denote
\[
\alpha_k=\int_0^t \EE\left[(h^2(s,X_s)+h^2(s,Y_s))\mathbbm{1}_{2^{-k-1}\leq|X_t-Y_t|\leq 2^{-k}}\right]\,ds.
\]
Note that
\[\begin{split}
\sum_k \alpha_k\leq &\int_0^t \EE\left((h^2(s,X_s)+h^2(s,Y_s)\right)\,ds\\
=&\int_0^t\int h^2(s,x)\,(u_X(dx,s)+u_Y(dx,s))\,ds\leq C.
\end{split}\]
Therefore $\alpha_k\longrightarrow 0$ as $k\rightarrow +\infty$.

Denote similarly
\[
\beta_k=\int_0^t \EE\left((\bar h_{2^{-k}}(s,X_s)+\bar h_{2^{-k}}(s,Y_s))\,\mathbbm{1}_{2^{-k-1}\leq|X_t-Y_t|\leq 2^{-k}}\right)\,ds.
\]
Denote $J_i=\{k,\ [2^{-k-1},\ 2^{-k})\subset I_i\}$. Note that $|J_i|\geq \frac{1}{C}\,|\log
b_i|$ (in fact, $|J_i|=\frac{|\log b_i|}{2\log 2}$) and since $\bar h_\eps$ is fixed on $\eps\in I_i$
\[\begin{split}
\frac{1}{|J_i|}\sum_{k\in J_i} \beta_k&\leq \frac{1}{|J_i|}\int_0^t \int \bar h_{b_i}(s,x)\,(u_X(dx,s)+u_Y(dx,s))\,ds\\
&\leq \frac{C^2}{b_i\,\phi(b_i^{-1})}\longrightarrow 0\quad \mbox{as}\ i\rightarrow \infty.
\end{split}\] 
Therefore $\beta_{n_k}\longrightarrow 0$ as $k\rightarrow +\infty$ for some subsequence $n_k\rightarrow+\infty$.
Consequently, since the sequence of functions $L_\eps$ is non increasing,
\[
\sup_{t\in [0,\ T]} \EE(L_\eps(X_t-Y_t))\longrightarrow 0\ \mbox{as}\ \eps\rightarrow 0.
\]
On the other hand
\[
\EE(L_\eps(X_t-Y_t))\geq \PP(|X_t-Y_t|>\eps),
\]
and by taking the limit $\eps\rightarrow 0$, we deduce that for any $t\in[0,T]$
\[
\PP(|X_t-Y_t|>0)=0.
\]
Therefore, $X_t=Y_t$ for all $t\in\QQ\cap[0,T]$ almost surely, and since $X_t$ and $Y_t$ have a.s.
continuous paths, we deduce that
$$
\PP(\sup_{t\in[0,T]}|X_t-Y_t|=0)=1,
$$
which proves pathwise uniqueness.
\section{Proof of Theorem \ref{theo1d}}
\label{sec:pf-1d}
This proof follows exactly the same steps as the general multi-dimensional case given in Section \ref{proofmaintheo}. The only differences are the functionals used and accordingly we skip the other parts of the proof which are identical.

Technically the reason why the one dimensional case is so special is that $|x|$ is linear
except at $x=0$. We do not know whether this corresponds to a deeper more intrinsic difference
between $d=1$ and $d>1$ or if the better results are in fact also true for $d>1$. 
\subsection{Existence}
For $d=1$, we replace the functional $Q^{(\varepsilon)}_{nm}$ by
\[
\tilde Q^{(\varepsilon)}_{nm}(t)=e^{-U_t^{n,m}}\,|X_t^n-X_t^m|\,\log \left(1+\frac{|X_t^n-X_t^m|^2}{\eps^2}\right),
\]
for $U_t^{n,m}$ a nonnegative random variable satisfying $dU_t^{n,m}=\lambda_t^{n,m}\,dt$ with
$\lambda_t^{n,m}$ an adapted process (measurable function of a continuous, adapted process) to
be chosen later.

Note that $f(x)=|x|\,\log(1+|x|^2/\eps^2)$ satisfies
$$
|f'(x)|\leq 4\log\left(1+\frac{|x|^2}{\varepsilon^2}\right)\qquad\text{and}\qquad |f''(x)|\leq\frac{C}{\varepsilon+|x|}.
$$
Therefore by It\^o's formula
\[\begin{split}
\EE(\tilde Q^{(\varepsilon)}_{nm}(t))\leq &C+C\,\int_0^t \EE\left(\frac{|\sigma(X_s^n)-\sigma(X_s^m)|^2}{\eps+|X_s^n-X_s^m|}\right)\,ds+\frac{\eta(n,m)}{\eps}\\
&+\int_0^t \EE\Bigg(|X_s^n-X_s^m|\,\log(1+|X_s^n-X_s^m|^2/\eps^2)\\
&\qquad\qquad\left(
4\,\frac{|F(s,X_s^n)-F(s,X_s^m)|}{|X_s^n-X_s^m|}-\lambda_t^{n,m}\right)\Bigg)
\,ds.
\end{split}\]
The first term is treated identically as for the multi-dimensional case.
The only difference here is that the careful choice of $\tilde Q^{(\varepsilon)}_{nm}$ improved the exponent of $|X_s^n-X_s^m|$ to $1$ instead of $2$
in the denominator. Therefore this term can be controlled with the ${H^{1/2}(u_{n,m})}$ norm of $\sigma$ by using Lemma \ref{maximal1/2} instead of estimate \eqref{maximal}.

The drawback is that the term with $F$ must be dealt with differently. We introduce $\tilde h=M |\nabla F|$  s.t. 
\[
\int_0^T \int_{\RR^d} \tilde h(t,x)\,(u_m(t,dx)+u_n(t,dx))\,dt\leq C.
\]
One poses 
\[
\lambda_t^{n,m}=4\,\left(\tilde h(t,X_s^m)
+\tilde h(t,X_t^m)\right).
\]
Therefore we deduce that
\[
\sup_{t\leq T} \EE(\tilde Q^{(\varepsilon)}_{nm}(t))\leq C+\frac{\eta(n,m)}{\eps}.
\]
Using a similar method as in Theorem~\ref{maintheo}, we write for constants $L$ and $K$ to be chosen later
\begin{multline*}
  \EE(|X^n_t-X^m_t|^p) \leq\EE(|X^n_t-X^m_t|^p;\,|X^n_t-X^m_t|\geq L)+\frac{1}{|\log\varepsilon|^{p/2}} \\ +\PP(U^{n,m}_t\geq\log K)
+L^p\PP\left(|X^n_t-X^m_t|\geq\frac{1}{\sqrt{|\log\varepsilon|}};\,U^{n,m}_t\leq \log K\right)
\end{multline*}
 Note that
\[
\EE(U_t^{n,m})=\EE\left(\int_0^t \lambda_s^{n,m}\,ds\right)\leq \int_0^t \tilde h(s,x)\,(u_n(s,dx)+u_m(s,dx))\,ds\leq C.
\]
Consequently
\begin{equation*}
\PP(U_t^{n,m}\geq \log(K))\leq \frac{C}{\log K}.
\end{equation*}
In addition, for $\varepsilon$ small enough,
$$
\PP\left(|X^n_t-X^m_t|\geq\frac{1}{\sqrt{|\log\varepsilon|}};\,U^{n,m}_t\leq \log K\right)\leq\frac{K\,\EE\tilde
  Q^{(\varepsilon)}_{nm}(t)}{2\sqrt{|\log\varepsilon|}}.
$$
Therefore,
\begin{multline*}
  \EE(|X^n_t-X^m_t|^p) \leq C\left(\frac{1}{L}+\frac{1}{|\log\varepsilon|^{p/2}}+\frac{1}{\log K}+\frac{L^p\, K\,\left(1+\frac{\eta(n,m)}{\varepsilon}\right)}{\sqrt{|\log\varepsilon|}}\right).
\end{multline*}
Taking for example $\varepsilon=\eta(n,m)$, $K=|\log\varepsilon|^{1/8}$ and
$L=|\log\varepsilon|^{1/8p}$, we deduce that
$$
\sup_{t\in[0,T]}\EE(|X^n_t-X^m_t|^p)\rightarrow 0,\quad\text{\ as\ }n,m\rightarrow+\infty.
$$
The rest of the proof is similar.
\subsection{Uniqueness}
For simplicity, we assume here that $F=0$. Otherwise it is necessary to introduce $U_t$ as in the previous subsection but it is handled in exactly the same way.

We similarly change the definition of $L_\eps$ in
\[
\tilde L_\eps(x)=
|x|\ \mbox{if}\ |x|\geq \eps,\quad \tilde L_\eps(x)=0\ \mbox{if}\ |x|\leq \eps/2,\quad \|\nabla \tilde L_\eps\|_{L^\infty}+\eps\,\|\nabla^2 \tilde L_\eps\|\leq C,\\
\]
with $C$ independent of $\eps$. 

Applying It\^o's formula 
\[\begin{split}
&\EE(\tilde L_\eps(X_t-Y_t))\leq 
C\,\int_0^t \EE\left({\mathbbm{1}_{\eps/2\leq|X_t-Y_t|\leq \eps}}\,
\frac{|\sigma(X_s)-\sigma(Y_s)|^2}{\eps}\right)\,ds.
\end{split}\]
By using as before the assumptions, Lemma \ref{maximal1/2} and the corresponding definition of $H^{1/2}(u_X)$ and $H^{1/2}(u_Y)$, one deduces that
\[
\EE(\tilde L_\eps(X_t-Y_t))\longrightarrow 0\quad\mbox{as}\ \eps\rightarrow 0.
\]
This is slightly less strong than before ($L_\eps\gg\tilde L_\eps$ for $x\ll1$) but still enough. In particular one has if $\alpha\geq \eps$
\[
\PP(|X_t-Y_t|\geq \alpha)\leq \frac{1}{\alpha}\,\EE(\tilde L_\eps(X_t-Y_t)).
\] 
Therefore by taking $\eps\rightarrow 0$, one still obtains that for any $t\in [0,\ T]$,
\[
\PP(|X_t-Y_t|>0)=0.
\]
The conclusion follows in exactly the same way as before.
\section{Proof of Prop. \ref{apriorilp}}
\label{sec:pf-prop-apriori}
We simply use the energy estimates. The computations below are formal but could easily be made rigorous by taking a regularization of $\sigma, F$ and hence $a$ and then pass to the limit. 
\[\begin{split}
\frac{d}{dt}\int u^\alpha(t,x)\,dx=&-\alpha\,(\alpha-1)\int u^{\alpha-1}(t,x)\nabla u(t,x)\cdot F(t,x)\,dx \\ & -\alpha\,(\alpha-1)\,\int u^{\alpha-2}(t,x) \nabla u(t,x)^T\,a(t,x)\,\nabla u(t,x)\,dx\\
&-\alpha\,(\alpha-1)\int u^{\alpha-1}(t,x)\,\sum_{1\leq i,j\leq d}\frac{\partial u(t,x)}{\partial x_i}\,\frac{\partial
  a_{ij}(t,x)}{\partial x_j}\,dx. 
\end{split}\] 
Note that by \eqref{ellipticity}
\[
\int u^{\alpha-2}(t,x) \nabla u(t,x)^T\,a(t,x)\,\nabla u(t,x)\,dx\geq C\,\|\nabla u^{\alpha/2}\|_{L^2}^2.
\]
On the other hand 
\[\begin{split}
\int  u^{\alpha-1}(t,x)\nabla u(t,x)\cdot F(t,x)\,dx & \leq \|\nabla u^{\alpha/2}\|_{L^2}\,\|u^{\alpha/2}\|_{L^2}\,\|F\|_{L^\infty} \\
&\leq \frac{C}{4}\,\|\nabla u^{\alpha/2}\|_{L^2}^2+C'\,\int u^\alpha(t,x)\,dx.
\end{split}\]
And
\[\begin{split}
\int u^{\alpha-1}(t,x)\,\sum_{1\leq i,j\leq d}\frac{\partial u(t,x)}{\partial x_i}\,\frac{\partial
  a_{ij}(t,x)}{\partial x_j}\,dx &\leq \|\nabla u^{\alpha/2}\|_{L^2}\,\|u^{\alpha/2}\,\nabla a\|_{L^2}\\
&\leq \|\nabla u^{\alpha/2}\|_{L^2}\,\|\nabla a\|_{L^p}\,\|u^{\alpha/2}\|_{L^r},
\end{split}\]
with $1/2=1/p+1/r$, which can be done since $p>d\geq 2$ . Now by Sobolev embedding
\[
\|u^{\alpha/2}\|_{L^r}\leq \left(\int u^\alpha\,dx\right)^{\theta/2}\,\|\nabla u^{\alpha/2}\|_{L^2}^{1-\theta},
\]
for some $\theta\in (0,\ 1]$, precisely $1/r=1/2-(1-\theta)/d$ or $(1-\theta)/d=1/p$, provided that $p>d$. In that case we immediately deduce that
\[
\frac{d}{dt}\int u^\alpha(t,x)\,dx+\frac{C}{2}\int |\nabla u^{\alpha/2}|^2\,dx\leq C''\left(1+\|\nabla a\|_{L^p}^{2/\theta}\right)\,\int u^\alpha\,dx. 
\] 
This concludes the bound provided that
\[
\int_0^T \|\nabla a\|_{L^p}^{2/\theta}<\infty,
\]
which means that $\nabla a\in L^q_{t,loc}(L^p_x)$ with $1/q=\theta/2=1/2-d/2p$. This exactly corresponds to the condition $2/q+d/p=1$ with $p>d$.

Note that $p=d$ is critical here in the sense that the result could still hold in that case provided that the norm of $\nabla a$ is small enough with respect to the constant of ellipticity. 

Finally we hence deduce that for any $t$ and any $\alpha<\infty$
\[
\|u(t,.)\|_{L^\alpha}\leq \|u(t=0,.)\|_{L^\alpha}\leq C,
\]
with $C$ independent of $\alpha$ since $u_0\in L^1\cap L^\infty$. This implies that $\|u(t,.)\|_{L^\infty}\leq C$ and finishes the proof.

\section{Proof of Prop. \ref{prop:ae-well-posedness}}
\label{sec:pf-prop-ae-wp}
We are going to prove this result under the assumptions of Corollary~\ref{cor:better-v2}. The other cases are completely similar.

Fix a complete filtered probability space $(\Omega,(\mathcal{F}_t)_{t\geq 0},\PP)$ equipped with a $r$-dimensional standard Brownian
motion $W$. Fix also $u_0>0$ in $L^1\cap L^\infty$. Then, by Corollary~\ref{cor:better}, on the probability space
$(\RR^d\times\Omega,(\mathcal{B}(\RR^d)\otimes\mathcal{F}_t)_{t\geq 0},u_0(x)dx\,\PP(d\omega))$, there is strong existence and
pathwise uniqueness for~\eqref{sde} with $\xi(x,\omega)=x$. We deduce that strong existence for almost every deterministic initial
condition holds for \eqref{sde} on $(\Omega,(\mathcal{F}_t)_{t\geq 0},(W_t)_{t\geq 0},\PP)$.

For uniqueness, the two key points are
\begin{itemize}
\item first, that $u\in L^\infty$, $\sigma\in H^1(u)$ and $F\in W^{\phi,weak}(u)$ (instead of uniform bounds for $u_n$ only as in
  Theorem~\ref{maintheo});
\item second, that we are always in cases where uniqueness in law is known for \emph{all} initial conditions in~\eqref{sde}, and in
  particular for all deterministic initial conditions.
\end{itemize}
For all $x$ such that strong existence holds for~\eqref{sde} with $\xi=x$, let $X^x_t$ and $\hat{X}^x_t$ be two strong solutions
of~\eqref{sde} such that $X^x_0=\hat{X}^x_0=x$ a.s. Repeating the proof of Lemma~\ref{lem:1}, we have
\begin{multline*}
  \mathbb{E}\log\left(1+\frac{|X^x_t-\hat{X}^x_t|^2}{\varepsilon^2}\right)\leq
  C\int_0^t\mathbb{E}\left[ M|\nabla\sigma|(s,X^x_s)^2+M|\nabla\sigma|(s,\hat{X}^x_s)^2\right]ds \\
  +C\int_0^t\mathbb{E}\left[ (|F|+M_{1/\varepsilon}\nabla F)(s,X^x_s)+(|F|+M_{1/\varepsilon}\nabla F) (s,\hat{X}^x_s)\right]ds
\end{multline*}
By uniqueness in law, for all $s\geq 0$, $X^x_s$ and $\hat{X}^x_s$ have the same distribution, and so
\[\begin{split}
&\mathbb{E}\log\left(1+\frac{|X^x_t-\hat{X}^x_t|^2}{\varepsilon^2}\right)\\
&\qquad\qquad\leq
C\int_0^t\mathbb{E}\left[\left((M|\nabla\sigma|)^2+|F|+M_{1/\varepsilon}\nabla F\right)(s,X^x_s)\right]ds.
\end{split}\]
Let us denote by $M_t^\varepsilon(x)$ the integral in the r.h.s. Note that the l.h.s. may not be a measurable function of
$x$, but $M_t^\varepsilon(x)$ is, and choosing $\phi$ as in the proof of Corollary~\ref{maincorollary}
\begin{align*}
\int_{\RR^d} M_t^\varepsilon(x)\,u_0(x)\,ds & =\int_0^t\int_{\RR^d}\left((M|\nabla\sigma|)^2+|F|+M_{1/\varepsilon}\nabla F\right)(s,x)u(s,dx)\,ds \\ & 
\leq C\left(1+\frac{|\log\varepsilon|}{\varepsilon\phi(\varepsilon^{-1})}\right). 
\end{align*}
Now, copying the proof of Lemma~\ref{lem:2},
$$
\mathbb{E}(|X^x_t-\hat{X}^x_t|) \leq C\left[\sqrt{\varepsilon}+\frac{1}{L}+\frac{LM^\varepsilon_t(x)}{|\log\varepsilon|}\right]
$$
Let us denote by $N^\varepsilon_t(x)$ the r.h.s. Choosing $L=\left(\frac{1}{|\log\varepsilon|}+\tilde{\eta}(\varepsilon)\right)^{-1}$
with $\eta(\varepsilon)=(\varepsilon\phi(\varepsilon^{-1}))^{-1}$, we obtain
$$
\int_{\RR^d} N_t^\varepsilon(x)\,u_0(x)\,ds\leq C\left(\sqrt{\varepsilon}+\sqrt{\frac{1}{|\log\varepsilon|}+\tilde{\eta}(\varepsilon)}\right).
$$
Since the r.h.s.\ converges to 0 when $\varepsilon\rightarrow 0$, there exists a sequence $\varepsilon_k\rightarrow 0$ such that
$N_t^{\varepsilon_k}(x)\rightarrow 0$ for almost all $x$. The diagonal procedure then shows the existence of a subsequence
$\varepsilon'_k\rightarrow 0$ such that $N_t^{\varepsilon'_k}(x)\rightarrow 0$ for almost all $x$ and for all $t$ in a dense
denumerable subset of $[0,T]$. Since the paths of $X^x$ and $\hat{X}^x$ are continuous, we deduce that pathwise uniqueness holds for
almost all $x\in\mathbb{R}^d$.
\section*{Appendix: Sketch of the proof of Corollary \ref{maincorollary}}
\label{sec:pf-cor}
%
The only thing left to prove after Theorem \ref{maintheo} is: Assume  
$u\in L^{q'}_{t,loc}(L^{p'}_x(\R^d))$ then show that, for some super linear $\phi$,
\[
\|\sigma\|_{H^1(u)}\leq C\,\|\nabla\sigma\|_{L^{2q}_{t,loc}(L^{2p}_x)},\qquad
\|F\|_{W^{\phi,weak}(u)}\leq C\,\|\nabla F\|_{L^{q}_{t,loc}(L^{p}_x)}.
\]
From the fact that the maximal operator $M$ is bounded on $L^p$, $p>1$, then this is straightforward for $\sigma$ (as
$2p\geq 2>1$). It is also the case for $F$ whenever $p>1$ (taking $\phi(M)=M\log M$).

Therefore the key point is how to prove that for $F$ when $p=1$. 
 
Now fix $L$ and denote 
\[
h(t,x)=M_L\nabla F=C\,\sqrt{\log L}+C\,\int_{B_1(x)} \frac{|\nabla F(t,z)|\,\mathbbm{1}_{|\nabla F|\geq \sqrt{\log L}}\,dz}{(L^{-1}+|x-z|)\,|x-z|^{d-1}},
\] 
where $B_1(x)$ is the ball of radius 1 centered at $x$.
As $p'=\infty$, for almost any fixed $t$, $u(t,\cdot)\in L^1\cap
L^\infty_x$ and hence
\[\begin{split}
\int h(t,x)\,u(t,x)\,dx&\leq C\,\sqrt{\log L}+C\,\|u(t,.)\|_{L^\infty}\\
&\qquad\int \int_{B_1(x)} \frac{|\nabla F(t,z)|\,\mathbbm{1}_{|\nabla F|\geq \sqrt{\log L}}\,dz}{(L^{-1}+|x-z|)\,|x-z|^{d-1}}\,dx\\
\leq &C\,\sqrt{\log L}+C\,\log L\,\|u(t,.)\|_{L^\infty}\,\|\nabla F(t,.)\,\mathbbm{1}_{|\nabla F|\geq \sqrt{\log L}}\|_{L^1},
\end{split}\]
by Fubini's theorem. 

Therefore integrating now in time, by H\"older's estimates
\[
\int_0^T\int h(t,x)\,u(t,x)\,dx\,dt\leq C\,\sqrt{\log L}\,T+C\,\log L\,\|\nabla F\,\mathbbm{1}_{|\nabla F|\geq \sqrt{\log L}}\|_{L^q_t(L^1_x)}.
\]
Now if $\nabla F\in L^q_t(L^1_x)$ then de la Vall\'ee Poussin classical integrability result means that there exists a super linear $\psi$ s.t.
\[
\|\psi(\nabla F)\|_{L^q_t(L^1_x)}<\infty.
\]
Consequently 
\[
\int_0^T\int h(t,x)\,u(t,x)\,dx\,dt\leq C\,\sqrt{\log L}\,T+C\,\frac{(\log L)^{3/2}}{\psi(\sqrt{\log L})}.
\]
We may conclude that $\|\nabla F\|_{W^{\phi,weak}}$ is bounded for $\phi$ defined by
\[
\frac{L}{\phi(L)}=\frac{C\,\sqrt{\log L}}{\log L}+\frac{C\,\sqrt{\log L}}{\psi(\sqrt{\log L})},
\]
which is hence also super linear.
\bigskip

{\bf Acknowledgements.} P-E Jabin was partially supported by the KI-Net research network, NSF Grant 1107444.


\end{document}